\documentclass[11pt]{article}

\usepackage[margin=1in]{geometry}
\usepackage{amsmath,amssymb,amsthm,mathtools}
\usepackage{algorithm}
\usepackage{algpseudocode}
\usepackage{stmaryrd}
\usepackage{microtype}
\usepackage{xcolor}
\usepackage{graphicx}
\usepackage{aliascnt}
\usepackage[colorlinks=true,linkcolor=blue!60!black,urlcolor=blue!60!black]{hyperref}
\usepackage[nameinlink]{cleveref}
\makeatletter
\renewcommand{\theHALG@line}{\thealgorithm.\arabic{ALG@line}}
\makeatother

\newcommand{\R}{\mathbb{R}}
\newcommand{\C}{\mathbb{C}}
\newcommand{\E}{\mathbb{E}}
\newcommand{\Prob}{\mathbb{P}}
\newcommand{\e}{\mathrm{e}}
\newcommand{\mat}[1]{\mathbf{#1}}

\newcommand{\col}{\mathord{:}}

\DeclareMathOperator{\diag}{diag}
\DeclareMathOperator{\rank}{rank}
\DeclareMathOperator{\tr}{tr}
\DeclareMathOperator*{\argmax}{arg\,max}

\DeclarePairedDelimiter{\abs}{\lvert}{\rvert}
\DeclarePairedDelimiter{\norm}{\lVert}{\rVert}

\numberwithin{equation}{section}

\theoremstyle{plain}
\newtheorem{theorem}{Theorem}[section]
\newaliascnt{lemma}{theorem}
\newtheorem{lemma}[lemma]{Lemma}
\aliascntresetthe{lemma}
\newaliascnt{proposition}{theorem}
\newtheorem{proposition}[proposition]{Proposition}
\aliascntresetthe{proposition}
\newaliascnt{corollary}{theorem}
\newtheorem{corollary}[corollary]{Corollary}
\aliascntresetthe{corollary}

\theoremstyle{definition}
\newaliascnt{definition}{theorem}

\aliascntresetthe{definition}

\theoremstyle{remark}
\newaliascnt{remark}{theorem}
\newtheorem{remark}[remark]{Remark}
\aliascntresetthe{remark}

\crefname{theorem}{theorem}{theorems}
\crefname{lemma}{lemma}{lemmas}
\crefname{proposition}{proposition}{propositions}
\crefname{corollary}{corollary}{corollaries}
\crefname{definition}{definition}{definitions}
\crefname{remark}{remark}{remarks}
\crefname{subsection}{subsection}{subsections}

\title{Convergence rates of randomly pivoted methods for low-rank approximation}
\author{Ryan Divan\thanks{Department of Mathematics, Princeton University,
Princeton, NJ 08544, United States. \href{mailto:rd2700@princeton.edu}{rd2700@princeton.edu},
\href{mailto:gilles@princeton.edu}{gilles@princeton.edu}.}
\and Marc Aur\`ele Gilles\footnotemark[1]}
\date{\today}

\begin{document}

\maketitle

\begin{abstract}
Randomly pivoted Cholesky, QR, and LU are iterative algorithms that form
structured low-rank approximations of a matrix by sampling columns, or
rows and columns, from the residual. We give convergence rates for these
methods that depend on the decay of the singular values of the original
matrix. The rates hold, up
to a constant, when pivots are drawn from an approximation to the exact
distribution.
This extra degree of freedom can be used to trade exact sampling for
cheaper approximations. We use it to establish rates for variants of
randomly pivoted LU, including one that picks pivots using a sketch of the input matrix.
\end{abstract}

\section{Introduction}
Pivoted Cholesky, QR, and LU are classical iterative methods for constructing
low-rank matrix approximations from selected pivot columns, or
pivot rows and columns. Traditionally, each iteration selects a pivot greedily
from the current residual: pivoted Cholesky and QR choose the largest diagonal
entry and column norm, respectively, while LU with complete pivoting chooses
the largest entry. More recently,
randomized variants have been introduced that sample pivots according to
these same quantities
\cite{FriezeKannanVempala2004,DeshpandeVempala2006,DeshpandeEtAl2006,
MuscoWoodruff2017,Poulson2020,ChenEtAl2025,GillesWilber2026}.\footnote{See
\cite[Section~3]{ChenEtAl2025} for a historical overview of
randomly pivoted QR and Cholesky.}
Randomized
pivoting can be more robust to outliers (e.g., a handful of large entries in the matrix) and can be considerably
cheaper when an exhaustive search over the residual is impractical.

In this work, we extend recently obtained convergence rates for greedy pivoting \cite{Gilles_convergence} to the randomized setting. 
Our analysis yields substantially stronger bounds for randomly pivoted LU
(RPLU) and its variants than the previous analysis~\cite{GillesWilber2026},
as well as a modest improvement over concurrent results for randomly pivoted
Cholesky (RPCholesky) and randomly pivoted QR
(RPQR)~\cite{Epperly2026}.
For all three algorithms, we obtain convergence rates that depend on the decay
of the spectrum. For example, if the singular values of the original matrix decay as $\mathcal{O}(\rho^k)$ for some $0 < \rho<1$, our bounds guarantee that the expected 
error decays as $\mathcal{O}(k\rho^{k/2})$ for RPCholesky and RPLU, and
$\mathcal{O}(\sqrt{k}\rho^{k/2})$ for RPQR.

The analysis also extends to approximate sampling, where pivots are drawn
from an approximation to the exact distribution, with only a constant-factor
loss in the convergence guarantees. This flexibility allows us to analyze
substantially more efficient pivoting schemes based on inexpensive estimates of
the residual row norms by relating them to RPLU.

The rest of the paper is organized as follows. After reviewing the necessary
background in~\cref{sec:background}, we analyze RPCholesky and RPQR
in~\cref{sec:qr-cholesky} and RPLU in~\cref{sec:lu}. We then apply the
approximate-pivoting framework to two efficient RPLU schemes
in~\cref{sec:efficient-rplu}.

\section{Background}
\label{sec:background}

\subsection{Notation}

Let $\mat{A}\in\C^{m\times n}$ have entries $\mat{A}_{ij}$. We use
MATLAB-style indexing, with a colon denoting all indices. Thus,
$\mat{A}_{\col,j}$ and $\mat{A}_{i,\col}$ denote the $j$th column and
$i$th row of $\mat{A}$, respectively. Let
$[n]:=\{1,\ldots,n\}$. For ordered index tuples
$\mathcal{I}\subset[m]$ and $\mathcal{J}\subset[n]$, we write
$\mat{A}_{\mathcal{I},\mathcal{J}}$ for the submatrix formed from the
corresponding rows and columns.

For a square matrix $\mat{H}\in\C^{n\times n}$, we write
$\mat{H}\succeq0$ to indicate that $\mat{H}$ is Hermitian positive semidefinite
(PSD). We denote the singular values of a general matrix
$\mat{A}\in\C^{m\times n}$ by
\[
  \sigma(\mat{A})
  :=
  \bigl(
    \sigma_1(\mat{A}),\ldots,
    \sigma_{\min\{m,n\}}(\mat{A})
  \bigr),
  \qquad
  \sigma_1(\mat{A})\geq\cdots\geq
  \sigma_{\min\{m,n\}}(\mat{A})\geq0.
\]
If $\mat{H}\succeq0$, its eigenvalues coincide with its singular values,
and we write
\[
  \lambda(\mat{H})
  :=
  \bigl(\lambda_1(\mat{H}),\ldots,\lambda_n(\mat{H})\bigr)
  =
  \sigma(\mat{H}).
\]

We use the Frobenius and nuclear norms
\[
  \norm{\mat{A}}_F
  :=
  \left(\sum_{i,j}\abs{\mat{A}_{ij}}^2\right)^{1/2},
  \qquad
  \norm{\mat{A}}_*
  :=
  \sum_j \sigma_j(\mat{A}).
\]
and recall that for $\mat{H} \succeq 0$, the nuclear norm is equal to its trace
$
  \norm{\mat{H}}_*
  =
  \tr\mat{H}
  :=
  \sum_i \mat{H}_{ii}.
$

Finally, we recall that the Eckart--Young--Mirsky theorem~\cite{Mirsky1960} states that,
for every $r<\min\{m,n\}$,
\[
  \min_{\rank(\mat{B})\leq r}
  \norm{\mat{A}-\mat{B}}_F
  =
  \left(\sum_{j>r}\sigma_j(\mat{A})^2\right)^{1/2},
  \qquad
  \min_{\rank(\mat{B})\leq r}
  \norm{\mat{A}-\mat{B}}_*
  =
  \sum_{j>r}\sigma_j(\mat{A}).
\]
Thus, the decay of the singular values provides a natural benchmark for
the quality of a low-rank approximation.

\subsection{Algorithms}
We state the three iterations producing low-rank approximations analyzed below.

\subsubsection*{Pivoted Cholesky}

Let $\mat{A}\in\C^{n\times n}$ be positive semidefinite.  Pivoted Cholesky
generates a sequence of approximants $\mat{\widehat A}^{(k)}$ and residuals
$\mat{A}^{(k)}$ through the iteration
\begin{equation}
\begin{aligned}
  \mat{A}^{(0)}&=\mat{A},
  &\mat{A}^{(k+1)}
  &=\mat{A}^{(k)}-
  \frac{\mat{A}_{\col,i_{k+1}}^{(k)}
        \mat{A}_{i_{k+1},\col}^{(k)}}
       {\mat{A}^{(k)}_{i_{k+1},i_{k+1}}},\\
  \mat{\widehat A}^{(0)}&=\mat{0},
  &\mat{\widehat A}^{(k+1)}
  &=\mat{\widehat A}^{(k)}+
  \frac{\mat{A}_{\col,i_{k+1}}^{(k)}
        \mat{A}_{i_{k+1},\col}^{(k)}}
       {\mat{A}^{(k)}_{i_{k+1},i_{k+1}}}.
\end{aligned}
\label{eq:cholesky-iteration}
\end{equation}
We refer to $i_{k+1}$ as the pivot index and
to
$\mat{A}^{(k)}_{i_{k+1},i_{k+1}}>0$ as the pivot.
Both the residuals $\mat{A}^{(k)}$ and approximants
$\mat{\widehat A}^{(k)}$ remain positive semidefinite. As long as the residual
is nonzero, each iteration increases the rank of the approximant by one and
decreases the rank of the residual by one. 
For all algorithms, if $\mat{A}^{(\ell)}=0$, we define
$
  \mat{A}^{(k)}=0,$
 and 
  $\mat{\widehat A}^{(k)}=\mat{A}$ for all 
  $k\geq\ell.$

The classical pivoted Cholesky algorithm is the greedy pivot choice
\cite{Higham1990}:
\[
  i_{k+1}\in\argmax_i \mat{A}^{(k)}_{ii}.
\]
The use of pivoted Cholesky for low-rank approximation has been studied
extensively; see, e.g.,~\cite{HarbrechtPetersSchneider2012,
CortinovisKressnerMassei2020,DeMarchiSchabackWendland2005,
SantinHaasdonk2017,JeongTownsend2025}.
Randomly pivoted Cholesky~\cite{ChenEtAl2025} instead samples
the pivot index from the distribution
\begin{equation}
\label{eq:rpcholesky-distribution}
  \Prob(i_{k+1}=i\mid \mat{A}^{(k)})
  =\frac{\mat{A}^{(k)}_{ii}}{\tr \mat{A}^{(k)}}.
\end{equation}

\subsubsection*{Pivoted QR}
For a general rectangular matrix $\mat{B}\in\C^{m\times n}$, pivoted QR generates a sequence of approximants $\mat{\widehat B}^{(k)}$ and residuals $\mat{B}^{(k)}$ through the iteration
\begin{equation}
\begin{aligned}
  \mat{B}^{(0)}&=\mat{B},
  &\mat{B}^{(k+1)}
  &=\mat{B}^{(k)}-
  \frac{\mat{B}_{\col,j_{k+1}}^{(k)}
        (\mat{B}_{\col,j_{k+1}}^{(k)})^{*}}
       {\norm{\mat{B}_{\col,j_{k+1}}^{(k)}}_2^2}
  \mat{B}^{(k)},\\
  \mat{\widehat B}^{(0)}&=\mat{0},
  &\mat{\widehat B}^{(k+1)}
  &=\mat{\widehat B}^{(k)}+
  \frac{\mat{B}_{\col,j_{k+1}}^{(k)}
        (\mat{B}_{\col,j_{k+1}}^{(k)})^{*}}
       {\norm{\mat{B}_{\col,j_{k+1}}^{(k)}}_2^2}
  \mat{B}^{(k)}.
\end{aligned}
\label{eq:qr-iteration}
\end{equation}
Here $j_{k+1}$ is the selected column pivot index.
The classical pivoted column QR algorithm is the greedy choice
\cite{BusingerGolub1965}:
\[
  j_{k+1}\in\argmax_j\norm{\mat{B}_{\col,j}^{(k)}}_2.
\]
Pivoted QR has a similarly long history as a method for low-rank approximation;
see, e.g.,~\cite{ChandrasekaranIpsen1994,KawamuraSuda2021,
BinevEtAl2011,DeVorePetrovaWojtaszczyk2013,AntilChenField2018}.
Randomly pivoted QR instead samples pivots from the distribution
\cite{FriezeKannanVempala2004,DeshpandeVempala2006,DeshpandeEtAl2006}
\begin{equation}
\label{eq:rpqr-distribution}
  \Prob(j_{k+1}=j\mid \mat{B}^{(k)})
  =\frac{\norm{\mat{B}_{\col,j}^{(k)}}_2^2}
         {\norm{\mat{B}^{(k)}}_F^2}.
\end{equation}

As with pivoted Cholesky, each step of the QR iteration
in~\cref{eq:qr-iteration} increases the rank of the approximant by one and
decreases the rank of the residual by one. The mechanisms are different:
Cholesky subtracts from the residual a rank-one outer product formed from the
pivot column and corresponding row, whereas QR projects the residual onto the
orthogonal complement of the pivot column. Nevertheless, the two algorithms
are intimately connected through the \textit{Gram correspondence}
\cite[Theorem~2.12]{Epperly2025},
which we recall below.

\begin{remark}[Gram correspondence]
\label{rem:gram-correspondence}
Let $\mat{B}\in\C^{m\times n}$ and let
$\mat{A}:=\mat{B}^{*}\mat{B}$ be its Gram matrix. Applying pivoted QR to
$\mat{B}$ and pivoted Cholesky to $\mat{A}$ with the same pivot sequence
then gives
\[
  \mat{A}^{(k)}=(\mat{B}^{(k)})^{*}\mat{B}^{(k)}.
\]

Furthermore,
\[
  \tr \mat{A}^{(k)}=\norm{\mat{B}^{(k)}}_F^2,
  \qquad
  \mat{A}^{(k)}_{jj}=\norm{\mat{B}_{\col,j}^{(k)}}_2^2.
\]
Thus, the sampling distributions of RPCholesky
in~\cref{eq:rpcholesky-distribution} and of RPQR
in~\cref{eq:rpqr-distribution} coincide under the Gram correspondence.
Therefore, any error bound for RPCholesky immediately gives the
corresponding result for RPQR.
\end{remark}

\subsubsection*{Pivoted LU}

Pivoted LU also applies to general rectangular matrices
$\mat{A}\in\C^{m\times n}$ but, like Cholesky, updates the residual by
subtracting a rank-one outer product rather than by orthogonal projection.
It generates a sequence of approximants $\mat{\widehat A}^{(k)}$ and
residuals $\mat{A}^{(k)}$ through the Gaussian elimination iteration
\begin{equation}
\begin{aligned}
  \mat{A}^{(0)}&=\mat{A},
  &\mat{A}^{(k+1)}
  &=\mat{A}^{(k)}-
  \frac{\mat{A}_{\col,j_{k+1}}^{(k)}
        \mat{A}_{i_{k+1},\col}^{(k)}}
       {\mat{A}_{i_{k+1},j_{k+1}}^{(k)}},\\
  \mat{\widehat A}^{(0)}&=\mat{0},
  &\mat{\widehat A}^{(k+1)}
  &=\mat{\widehat A}^{(k)}+
  \frac{\mat{A}_{\col,j_{k+1}}^{(k)}
        \mat{A}_{i_{k+1},\col}^{(k)}}
       {\mat{A}_{i_{k+1},j_{k+1}}^{(k)}}.
\end{aligned}
\label{eq:rplu-iteration}
\end{equation}
Here $(i_{k+1},j_{k+1})$ is the pivot location, and
$\mat{A}^{(k)}_{i_{k+1},j_{k+1}}\neq0$ is the pivot. 
For an initial positive semidefinite matrix $\mat{A}$, restricting the pivots to the diagonal
reduces the LU iteration to the Cholesky iteration
in~\cref{eq:cholesky-iteration}.

The classical LU with complete pivoting (CPLU) algorithm is the greedy
pivot choice \cite{Wilkinson1961}
\[
  (i_{k+1},j_{k+1})
  \in
  \argmax_{i,j}\abs{\mat{A}^{(k)}_{ij}}.
\]
The low-rank approximation quality of CPLU has been studied
in~\cite{Bebendorf2000,TownsendTrefethen2015,
CortinovisKressnerMassei2020,Gilles_convergence}.
Randomly pivoted LU~\cite{GillesWilber2026} instead samples the pivot
location from the distribution
\begin{equation}
\label{eq:rplu-distribution}
  \Prob((i_{k+1},j_{k+1})=(i,j)\mid \mat{A}^{(k)})
  =
  \frac{\abs{\mat{A}_{ij}^{(k)}}^2}
       {\norm{\mat{A}^{(k)}}_F^2}.
\end{equation}

\subsubsection*{Approximate sampling}

For all three algorithms, we also allow pivots to be sampled from approximations to the exact
distributions defined above. Let $\mathcal{P}$ denote the set of admissible
pivot indices: $\mathcal{P}=[n]$ for RPCholesky and RPQR, and
$\mathcal{P}=[m]\times[n]$ for RPLU. At iteration $k$, let
\[
\pi_p^{(k)}
:=
\Prob\bigl(p_{k+1}=p\mid\mat{R}^{(k)}\bigr)
\]
denote the corresponding exact sampling probability
in~\cref{eq:rpcholesky-distribution,eq:rpqr-distribution,eq:rplu-distribution}.
Conditional on the current residual $\mat{R}^{(k)}$, let $q_p^{(k)}$ be another probability distribution on $\mathcal{P}$:
\[
  q_{p}^{(k)}\geq0,
  \qquad
  \sum_{p \in \mathcal{P}}q_{p}^{(k)}=1.
\]
For $0<\gamma\leq1$, we call the distribution
$q^{(k)}$ a $\gamma$-approximate sampling
distribution if
\begin{equation}
\label{eq:approximate-sampling}
  q_{p}^{(k)}
  \leq
  \frac{1}{\gamma}\pi_p^{(k)}
  \qquad
  \text{for every }p\in\mathcal{P}.
\end{equation}
Exact sampling
corresponds to $\gamma=1$.

\subsection{Determinant identities and elementary symmetric polynomials}

The analysis of all three algorithms relies on two classical determinant
identities. The first relates the successive LU pivots to a minor of the original
matrix; see, e.g.,~\cite[Equation~(4.2)]{Wilkinson1961}. The same identity
underlies the analysis of greedy pivoting in~\cite{Gilles_convergence}.

\begin{lemma}[Pivot-product identity]
\label{lem:pivot-product}
Let $\mat{A}^{(\ell)}$ be the residuals generated by the pivoted LU
iteration~\eqref{eq:rplu-iteration} using the pivot sequence
$(i_\ell,j_\ell)$, $\ell=1,\ldots,k$. Set $\mathcal{I}:=(i_1,\ldots,i_k)$ and
$\mathcal{J}:=(j_1,\ldots,j_k)$. Then,
\[
  \det\mat{A}_{\mathcal{I},\mathcal{J}}
  =
  \prod_{\ell=1}^k
  \mat{A}^{(\ell-1)}_{i_\ell,j_\ell}.
\]
\end{lemma}

To state the second identity, define the degree-$k$ elementary symmetric
polynomial by
\[
  e_k(\lambda_1,\ldots,\lambda_n)
  :=
  \sum_{1\leq i_1<\cdots<i_k\leq n}
  \prod_{\ell=1}^k \lambda_{i_\ell}.
\]
If $\mat{H}\in\C^{n\times n}$ has eigenvalues
$\lambda_1(\mat{H}),\ldots,\lambda_n(\mat{H})$, we write
\[
  e_k(\lambda(\mat{H}))
  :=
  e_k\bigl(
    \lambda_1(\mat{H}),\ldots,\lambda_n(\mat{H})
  \bigr).
\]
Elementary symmetric polynomials arise naturally in determinant identities;
see, e.g.,~\cite[Section~1.2]{HornJohnson2013}.

The second identity relates these polynomials to the principal minors of a
matrix~\cite[Theorem~1.2.16]{HornJohnson2013}.

\begin{lemma}[Principal-minor identity]
\label{lem:principal-minor-identity}
Let $\mat{H}\in\C^{n\times n}$. For every $k\in\{0,\ldots,n\}$,
\[
  e_k(\lambda(\mat{H}))
  =
  \sum_{\substack{I\subseteq[n]\\|I|=k}}
  \det\mat{H}_{I,I}.
\]
\end{lemma}
This identity also plays an important role in the analysis of volume sampling
and DPP-based low-rank approximation
\cite{DeshpandeEtAl2006,GuruswamiSinop2012,LiJegelkaSra2016}. Both determinant
identities are also used in the concurrent work~\cite{Epperly2026}, which
establishes oversampled bounds for RPCholesky (and thus RPQR).

\section{Convergence rates for RPCholesky and RPQR}
\label{sec:qr-cholesky}

We begin by combining the two determinant identities above to bound the
expected RPCholesky error in terms of $e_k(\lambda(\mat{A}))$. We will then obtain rates by
estimating this quantity under algebraic and geometric eigenvalue decay.

\begin{theorem}[Approximate RPCholesky expected residual trace]
\label{lem:rpchol-det-moment}
Let $\mat{A}^{(0)}=\mat{A}\succeq0$, and let $\mat{A}^{(\ell)}$ be the
residuals generated by $\gamma$-approximate RPCholesky. Then, for every
integer $k\geq1$,
\[
  \E[\tr \mat{A}^{(k-1)}]\leq \gamma^{-1}\bigl(k!\,e_k(\lambda(\mat{A}))\bigr)^{1/k}.
\]
\end{theorem}

\begin{proof}
Let $\mathcal{L}_k$ denote the set of all ordered pivot paths
\[
  \mathcal{I}=(i_1,\ldots,i_k),
\]
where $i_\ell \in [n]$ and
$\mat{A}^{(\ell-1)}_{i_\ell,i_\ell}>0$ for $\ell=1,\ldots,k$.
The indices in each path are distinct, since a selected index has zero
diagonal entry at every subsequent iteration.

Fix $\mathcal{I} \in \mathcal{L}_k$. The RPCholesky rule
in~\cref{eq:rpcholesky-distribution} and the approximate sampling condition
in~\cref{eq:approximate-sampling} give
\[
  \Prob(\mathcal{I})
  \leq \gamma^{-k}
  \prod_{\ell=1}^k
  \frac{\mat{A}^{(\ell-1)}_{i_\ell,i_\ell}}
       {\tr\mat{A}^{(\ell-1)}}.
\]
Therefore
\[
\begin{aligned}
  \E\left[
    \prod_{\ell=0}^{k-1}\tr\mat{A}^{(\ell)}
  \right]
  &=\sum_{\mathcal{I} \in \mathcal{L}_k}
    \Prob(\mathcal{I})
    \prod_{\ell=0}^{k-1}\tr\mat{A}^{(\ell)}\\
  &\leq \gamma^{-k}
    \sum_{\mathcal{I} \in \mathcal{L}_k}
    \left(
      \prod_{\ell=1}^k
      \frac{\mat{A}^{(\ell-1)}_{i_\ell,i_\ell}}
           {\tr\mat{A}^{(\ell-1)}}
    \right)
    \left(
      \prod_{\ell=0}^{k-1}\tr\mat{A}^{(\ell)}
    \right)\\
  &=\gamma^{-k}
    \sum_{\mathcal{I} \in \mathcal{L}_k}
    \prod_{\ell=1}^k
    \mat{A}^{(\ell-1)}_{i_\ell,i_\ell}\\
  &=\gamma^{-k}
    \sum_{\mathcal{I} \in \mathcal{L}_k}
    \det \mat{A}_{\mathcal{I},\mathcal{I}},
\end{aligned}
\]
where the last line follows from~\cref{lem:pivot-product}.

Let $I=\{i_1,\ldots,i_k\}\subset[n]$ be the unordered set associated with
the ordered path $\mathcal{I}\in\mathcal{L}_k$. For each such set $I$, there
are at most $k!$ corresponding paths in $\mathcal{L}_k$, all satisfying
$\det \mat{A}_{\mathcal{I},\mathcal{I}}=\det \mat{A}_{I,I}$. Thus,

\[
\begin{aligned}
  \E\left[
    \prod_{\ell=0}^{k-1}\tr\mat{A}^{(\ell)}
  \right]
  & \le \gamma^{-k} \sum_{\mathcal{I} \in \mathcal{L}_k}\det \mat{A}_{\mathcal{I},\mathcal{I}}\\
  & \le \gamma^{-k} k! \sum_{\substack{I\subseteq[n]\\|I|=k}}\det \mat{A}_{I,I}\\
  & = \gamma^{-k} k!\,e_k(\lambda(\mat{A})),
\end{aligned}
\]
where the last line follows from \cref{lem:principal-minor-identity}.
For any pivot sequence,
$\tr \mat{A}^{(\ell+1)} \leq \tr \mat{A}^{(\ell)}$, thus,
\begin{equation}
\label{eq:rpchol-geometric-mean}
  \bigl(\tr\mat{A}^{(k-1)}\bigr)^k
  \leq
  \prod_{\ell=0}^{k-1}\tr\mat{A}^{(\ell)}.
\end{equation}
Finally, convexity of $x\mapsto x^k$ and Jensen's inequality give
\[
\begin{aligned}
  \bigl(\E[\tr\mat{A}^{(k-1)}]\bigr)^k
  &\leq
  \E\left[\bigl(\tr\mat{A}^{(k-1)}\bigr)^k\right]\\
  &\leq
  \E\left[
    \prod_{\ell=0}^{k-1}\tr\mat{A}^{(\ell)}
  \right]\\
  &\leq \gamma^{-k} k!\,e_k(\lambda(\mat{A})).
\end{aligned}
\]
Taking the $k$th root proves the claim.
\end{proof}

To use the pivot-product identity, we make one crude estimate
in~\cref{eq:rpchol-geometric-mean}: we effectively bound the final residual trace by the
geometric mean of the first $k$ residual traces.
The advantage is that Jensen's inequality is applied only once, rather than
at every iteration as in the original analysis of RPCholesky
in~\cite{ChenEtAl2025}. The resulting iterated one-step
bound cannot capture rapid geometric convergence: it remains subgeometric
regardless of how quickly the singular values decay; see
\cite[Section~3.3]{GillesWilber2026}. This limitation largely accounts for
the gap between the rates in the original analysis of
RPCholesky~\cite{ChenEtAl2025} and those established here and in the more
recent analysis~\cite{Epperly2026}.

To get error estimates from \cref{lem:rpchol-det-moment} under singular value decay,
we will use bounds on the corresponding decay of elementary symmetric polynomials.

\begin{lemma}[Estimates of elementary symmetric polynomials under decay]
\label{lem:esp-decay}
Let $(\lambda_j)_{j=1}^n$ be a nonnegative sequence and $1\leq k\leq n$. If
$\lambda_j\leq Cj^{-\alpha}$ for $\alpha>1$, then
\[
  e_k(\lambda_1,\ldots,\lambda_n)^{1/k}
  \leq
  C\left(\frac{\pi\e}{\alpha k\sin(\pi/\alpha)}\right)^\alpha.
\]
If instead $\lambda_j\leq C\rho^{j-1}$ for $0<\rho<1$, then
\[
  e_k(\lambda_1,\ldots,\lambda_n)^{1/k}
  \leq
  \frac{C}{1-\rho}\rho^{(k-1)/2}.
\]
\end{lemma}

The proof is deferred to~\cref{app:esp-decay}.
Using these estimates in~\cref{lem:rpchol-det-moment} gives the following
rates.
\begin{corollary}[Convergence rates for approximate RPCholesky]
\label{cor:rpchol-rates}
Let $\gamma$-approximate RPCholesky start from $\mat{A}=\mat{A}^{(0)} \succeq 0$. If $\lambda_j(\mat{A})\leq Cj^{-\alpha}$ for $\alpha>1$, then
\[
  \E\tr \mat{A}^{(k)}
  \leq
  \gamma^{-1}C\left(\frac{\pi\e}{\alpha\sin(\pi/\alpha)}\right)^\alpha
  (k+1)^{1-\alpha}.
\]
If instead $\lambda_j(\mat{A})\leq C\rho^{j-1}$ for $0<\rho<1$, then
\[
  \E\tr \mat{A}^{(k)}
  \leq
  \gamma^{-1}\frac{C(k+1)}{1-\rho}\rho^{k/2}.
\]
\end{corollary}

\begin{proof}
Apply \cref{lem:esp-decay} to the right hand-side in~\cref{lem:rpchol-det-moment}  and use
$((k+1)!)^{1/(k+1)}\leq k+1$.
\end{proof}

In~\cref{fig:rpchol-bounds}, we compare our bounds with the oversampled
$(r,\varepsilon)$ bound of~\cite{Epperly2026} and the doubling bound
of~\cite[Lemma~5.5]{ChenEtAl2025} under different decay scenarios, though we
note that the oversampled bound of~\cite{Epperly2026} is more general, and does not require any assumption on spectral
decay.

Under algebraic decay, the oversampled bound
agrees with our bound up to logarithmic factors. When the eigenvalues decay
as $\mathcal{O}(\rho^j)$, our bound is sharper. For example, taking
$\varepsilon\rightarrow1$, the
oversampled bound requires $k\geq6r$ once $r\geq10$,
corresponding at best to a decay rate of roughly
$\mathcal{O}(\rho^{k/6})$. In comparison, \cref{cor:rpchol-rates} gives
$\mathcal{O}(k\rho^{k/2})$. For very rapid spectral decay, however, the
doubling bound can be the strongest of the three, as seen in the
fastest-decay examples in~\cref{fig:rpchol-bounds}.

\begin{figure}[H]
  \centering
  \includegraphics[width=\linewidth]{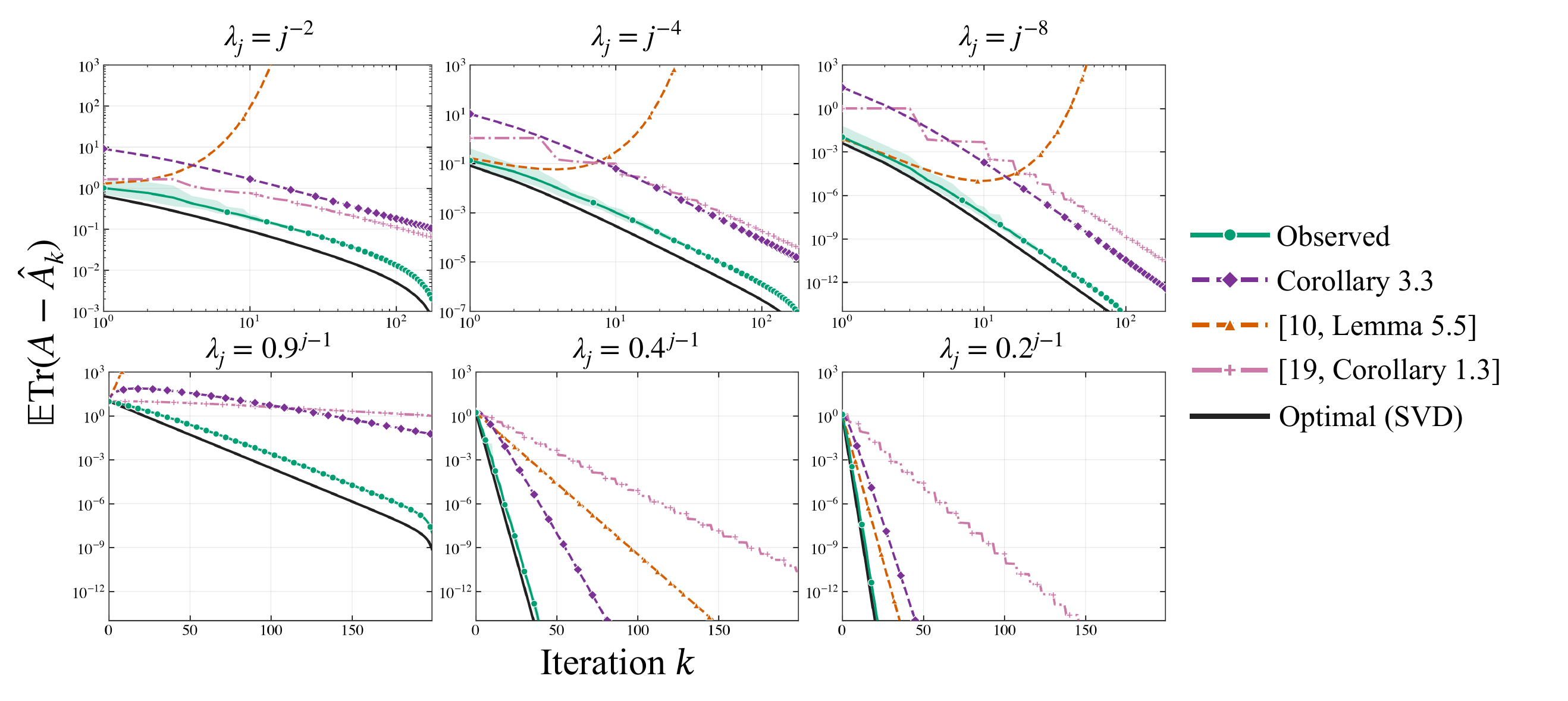}
  \caption{Comparison of the RPCholesky bound in~\cref{cor:rpchol-rates}
  with the bounds in~\cite[Lemma~5.5]{ChenEtAl2025} and
  \cite[Corollary~1.3]{Epperly2026} on random $200\times200$ PSD matrices. The top row uses
  algebraic decay $\lambda_j=j^{-\alpha}$ with
  $\alpha\in\{2,4,8\}$, and the bottom row uses geometric decay
  $\lambda_j=\rho^{j-1}$ with $\rho\in\{0.9,0.4,0.2\}$. For each decay
  profile, $\mat{A}=\mat{U}\diag(\lambda_1,\ldots,\lambda_{200})
  \mat{U}^{\top}$, where $\mat{U}$ is obtained from the QR factorization
  of a standard Gaussian matrix. The bound of~\cite{Epperly2026} is minimized over
  $(r,\varepsilon)$ by a grid search at each iteration. The observed curve
  is the mean of 30 RPCholesky trials with exact sampling ($\gamma=1$),
  with the minimum and maximum shaded; the black curve is the optimal
  rank-$k$ residual trace. }
  \label{fig:rpchol-bounds}

\end{figure}
The Gram correspondence immediately gives the corresponding decay rates for
$\gamma$-approximate RPQR.

\begin{corollary}[Convergence rates for approximate RPQR]
\label{cor:rpqr-rates}
Let $\gamma$-approximate RPQR start from $\mat{B}=\mat{B}^{(0)} \in \C^{m\times n}$. If $\sigma_j(\mat{B})\leq Cj^{-\alpha}$ for $\alpha>1/2$, then
\[
  \E\norm{\mat{B}^{(k)}}_F^2
  \leq
  \gamma^{-1}C^2\left(\frac{\pi\e}{2\alpha\sin(\pi/(2\alpha))}\right)^{2\alpha}
  (k+1)^{1-2\alpha}.
\]
If instead $\sigma_j(\mat{B})\leq C\rho^{j-1}$ for $0<\rho<1$, then
\[
  \E\norm{\mat{B}^{(k)}}_F^2
  \leq
  \gamma^{-1}\frac{C^2(k+1)}{1-\rho^2}\rho^k.
\]
\end{corollary}

\begin{proof}
Apply \cref{cor:rpchol-rates} to $\mat{B}^{*}\mat{B}$ and use
\cref{rem:gram-correspondence}, noting that
$\lambda_j(\mat{B}^{*}\mat{B})=\sigma_j(\mat{B})^2$.
\end{proof}

\section{Convergence rates for Randomly pivoted LU}
\label{sec:lu}

We now turn to randomly pivoted LU, defined by the iteration
in~\cref{eq:rplu-iteration}. The argument used for RPCholesky in the
previous section extends to RPLU with a few changes. The main difference is
that, for a general matrix, the residual norm need not decrease. The
pivot-product argument therefore controls the smallest residual among the
first $k$ iterates, rather than the final residual.

We begin with the analogue of \cref{lem:rpchol-det-moment}. For
$\mat{A}\in\C^{m\times n}$, we use the shorthand
\[
  e_k(\sigma(\mat{A})^2)
  :=
  e_k\bigl(
    \sigma_1(\mat{A})^2,\ldots,
    \sigma_{\min\{m,n\}}(\mat{A})^2
  \bigr).
\]

\begin{theorem}[Approximate RPLU expected residual norm]
\label{lem:rplu-det-moment}
Let $\gamma$-approximate RPLU start from
$\mat{A}=\mat{A}^{(0)}$. Then, for every integer $k\geq1$,
\[
  \E\left[
    \min_{0\leq\ell<k}\norm{\mat{A}^{(\ell)}}_F^2
  \right]
  \leq
  \gamma^{-1}(k!)^{2/k}e_k(\sigma(\mat{A})^2)^{1/k}.
\]
\end{theorem}

\begin{proof}
Let $\mathcal{L}_k$ denote the set of ordered pivot paths
\[
  \mathcal{I}=(i_1,\ldots,i_k),
  \qquad
  \mathcal{J}=(j_1,\ldots,j_k),
\]
where $i_\ell\in[m]$, $j_\ell\in[n]$, and
$\mat{A}^{(\ell-1)}_{i_\ell,j_\ell}\neq0$ for every $\ell=1,\ldots,k$.
Each path in $\mathcal{L}_k$ has no repeated row or column indices, since a selected row and column remain zero at every subsequent iteration.

Fix $(\mathcal{I},\mathcal{J})\in\mathcal{L}_k$. Then the RPLU rule
in~\cref{eq:rplu-distribution} and the approximate sampling condition
in~\cref{eq:approximate-sampling} give
\[
  \Prob(\mathcal{I},\mathcal{J})
  \leq
  \gamma^{-k}
  \prod_{\ell=1}^k
  \frac{\abs{\mat{A}^{(\ell-1)}_{i_\ell,j_\ell}}^2}
       {\norm{\mat{A}^{(\ell-1)}}_F^2}.
\]
Therefore
\[
\begin{aligned}
  \E\left[
    \prod_{\ell=0}^{k-1}\norm{\mat{A}^{(\ell)}}_F^2
  \right]
  &=
  \sum_{(\mathcal{I},\mathcal{J})\in\mathcal{L}_k}
  \Prob(\mathcal{I},\mathcal{J})
  \prod_{\ell=0}^{k-1}\norm{\mat{A}^{(\ell)}}_F^2\\
  &\leq
  \gamma^{-k}
  \sum_{(\mathcal{I},\mathcal{J})\in\mathcal{L}_k}
  \left(
    \prod_{\ell=1}^k
    \frac{\abs{\mat{A}^{(\ell-1)}_{i_\ell,j_\ell}}^2}
         {\norm{\mat{A}^{(\ell-1)}}_F^2}
  \right)
  \left(
    \prod_{\ell=0}^{k-1}\norm{\mat{A}^{(\ell)}}_F^2
  \right)\\
  &=
  \gamma^{-k}
  \sum_{(\mathcal{I},\mathcal{J})\in\mathcal{L}_k}
  \prod_{\ell=1}^k
  \abs{\mat{A}^{(\ell-1)}_{i_\ell,j_\ell}}^2\\
  &=
  \gamma^{-k}
  \sum_{(\mathcal{I},\mathcal{J})\in\mathcal{L}_k}
  \abs{\det\mat{A}_{\mathcal{I},\mathcal{J}}}^2,
\end{aligned}
\]
where the last line follows from \cref{lem:pivot-product}. For each pair of
unordered sets $I\subseteq[m]$ and $J\subseteq[n]$ with $\abs{I}=\abs{J}=k$,
there are at most $(k!)^2$ corresponding ordered paths
$(\mathcal{I},\mathcal{J})\in\mathcal{L}_k$, and each contributes
$\abs{\det\mat{A}_{I,J}}^2$. Thus:
\[
\begin{aligned}
  \E\left[
    \prod_{\ell=0}^{k-1}\norm{\mat{A}^{(\ell)}}_F^2
  \right]
  &\leq
  \gamma^{-k}(k!)^2
  \sum_{\substack{I\subseteq[m]\\\abs{I}=k}}
  \sum_{\substack{J\subseteq[n]\\\abs{J}=k}}
  \abs{\det\mat{A}_{I,J}}^2\\
  &=
  \gamma^{-k}(k!)^2
  \sum_{\substack{J\subseteq[n]\\\abs{J}=k}}
  \det\bigl((\mat{A}^{*}\mat{A})_{J,J}\bigr)\\
  &=
  \gamma^{-k}(k!)^2e_k(\sigma(\mat{A})^2).
\end{aligned}
\]
where the second line follows from Cauchy--Binet and the third from
\cref{lem:principal-minor-identity}.

The LU residual norm need not decrease at every iteration, so we set
$Z_k:=\min_{0\leq\ell<k}\norm{\mat{A}^{(\ell)}}_F^2$. Then
\[
  Z_k^k
  \leq
  \prod_{\ell=0}^{k-1}\norm{\mat{A}^{(\ell)}}_F^2.
\]
Convexity of $x\mapsto x^k$ then gives, by Jensen's inequality,
\[
  \bigl(\E[Z_k]\bigr)^k
  \leq
  \E[Z_k^k]
  \leq
  \gamma^{-k}(k!)^2e_k(\sigma(\mat{A})^2).
\]
Taking the $k$th root proves the result.
\end{proof}

We now combine this estimate with \cref{lem:esp-decay} to get convergence rates, exactly as for
RPCholesky.

\begin{corollary}[Approximate RPLU rates]
\label{cor:rplu-rates}
Let $\gamma$-approximate RPLU start from
$\mat{A}=\mat{A}^{(0)}$. If $
  \sigma_j(\mat{A})\leq Cj^{-\alpha}$ for $\alpha>1$, then
\[
  \E\min_{0\leq\ell\leq k}\norm{\mat{A}^{(\ell)}}_F^2
  \leq
  \gamma^{-1}C^2
  \left(\frac{\pi\e}{2\alpha\sin(\pi/(2\alpha))}\right)^{2\alpha}
  (k+1)^{2(1-\alpha)}.
\]
If instead $\sigma_j(\mat{A})\leq C\rho^{j-1}$ for $0 < \rho < 1$, then
\[
  \E\min_{0\leq\ell\leq k}\norm{\mat{A}^{(\ell)}}_F^2
  \leq
  \gamma^{-1}\frac{C^2(k+1)^2}{1-\rho^2}\rho^k.
\]
\end{corollary}

\begin{proof}
Apply \cref{lem:esp-decay} to
$\sigma_j(\mat{A})^2\leq C^2j^{-2\alpha}$ in the algebraic case and to
$\sigma_j(\mat{A})^2\leq C^2(\rho^2)^{j-1}$ in the geometric case.
Substitution into \cref{lem:rplu-det-moment}, together with
$((k+1)!)^{2/(k+1)}\leq(k+1)^2$, gives both results.
\end{proof}

Compared with the corresponding rates for RPQR in~\cref{cor:rpqr-rates}, the RPLU bounds incur an additional factor of
$k+1$. This comes from the $(k!)^2$ possible orderings of the row and
column pivots, in place of the single $k!$ factor for RPCholesky.

In practice, an important distinction between RPLU and RPCholesky is the cost of sampling a pivot.
For RPCholesky, forming the exact distribution requires only the $n$
diagonal entries of the residual, whereas exact RPLU sampling requires all
$mn$ entries and may therefore cost $\mathcal{O}(mn)$ per iteration.
Approximate pivoting allows us to replace this distribution by a cheaper
surrogate, while changing the bounds only by the factor $\gamma^{-1}$. In the next section, we use this flexibility to develop two efficient RPLU schemes: one tailored to Cauchy-like matrices and the other based on random sketching.

\section{Efficient approximate RPLU schemes}
\label{sec:efficient-rplu}

If the residual row norms are available, the RPLU distribution can be
sampled in two stages: first sample a row proportionally to its squared norm,
then form the selected row and sample a column proportionally to the squared
magnitudes of its entries~\cite{GillesWilber2026}:
\[
  \Prob(i_{k+1}=i\mid\mat{A}^{(k)})
  \Prob(j_{k+1}=j\mid i_{k+1}=i,\mat{A}^{(k)})
  =
  \frac{\norm{\mat{A}^{(k)}_{i,\col}}_2^2}
       {\norm{\mat{A}^{(k)}}_F^2}
  \frac{\abs{\mat{A}^{(k)}_{ij}}^2}
       {\norm{\mat{A}^{(k)}_{i,\col}}_2^2}
  =
  \frac{\abs{\mat{A}^{(k)}_{ij}}^2}
       {\norm{\mat{A}^{(k)}}_F^2}.
\]
Thus, once the residual row norms are known, exact RPLU pivots can be sampled
without forming the full distribution.
This factorization suggests a simple
design principle: approximate the distribution over rows, but retain
exact sampling within the selected row. We develop two instances of this
idea.
The
first applies to a very specialized class of matrices, while the second
applies to general matrices using sketching.

\subsection{Random Generator Pivoting}
\label{sec:rgp}
We first consider a pivoting distribution tailored to Cauchy-like matrices, a structured class for
which LU-based low-rank approximation can be implemented very
efficiently~\cite{GillesWilber2026}. A matrix
$\mat{A}\in\mathbb C^{m\times n}$ is called Cauchy-like if, for some $p$,
its entries can be written as
\begin{equation}
\label{eq:cauchy-like}
    \mat{A}_{ij}
    =
    \sum_{r=1}^p \frac{\mat{G}_{ir}\mat{B}_{rj}}{x_i-y_j}.
\end{equation}
Here, $x_1,\ldots,x_m\in\mathbb C$ and
$y_1,\ldots,y_n\in\mathbb C$ are the nodes and satisfy $x_i\neq y_j$ for
every $i,j$, while $\mat{G}\in\mathbb C^{m\times p}$ and
$\mat{B}\in\mathbb C^{p\times n}$ are the generators. The smallest possible
$p$ is called the displacement rank of $\mat{A}$.

Low-rank approximation of such matrices is an important computational step
in several applications, including matrix
equations~\cite{druskin2011analysis}, PDE
solvers~\cite{TownsendWilber2018}, and rational
approximation~\cite{antoulas2010interpolatory,aaa}.

Writing
\[
    \mat{X}:=\diag(x_1,\ldots,x_m),
    \qquad
    \mat{Y}:=\diag(y_1,\ldots,y_n),
\]
the entrywise representation in~\cref{eq:cauchy-like} is equivalent to the
Sylvester displacement equation
\[
    \mat{X}\mat{A}-\mat{A}\mat{Y}
    =
    \mat{G}\mat{B}.
\]
The key property is that the Cauchy-like structure is preserved by
Gaussian elimination. Suppose that
\[
    \mat{X}\mat{A}^{(k)}-\mat{A}^{(k)}\mat{Y}
    =
    \mat{G}^{(k)}\mat{B}^{(k)}.
\]
The Schur-complement update and the corresponding generator updates are
\begin{equation}
\begin{aligned}
    \mat{A}^{(k+1)}
    &:=
    \mat{A}^{(k)}-
    \frac{\mat{A}_{\col,j}^{(k)}\mat{A}_{i,\col}^{(k)}}
         {\mat{A}_{ij}^{(k)}},\\
    \mat{G}^{(k+1)}
    &:=
    \mat{G}^{(k)}-
    \frac{\mat{A}_{\col,j}^{(k)}\mat{G}_{i,\col}^{(k)}}
         {\mat{A}_{ij}^{(k)}},\\
    \mat{B}^{(k+1)}
    &:=
    \mat{B}^{(k)}-
    \frac{\mat{B}_{\col,j}^{(k)}\mat{A}_{i,\col}^{(k)}}
         {\mat{A}_{ij}^{(k)}}.
\end{aligned}
\label{eq:cauchy-generator-update}
\end{equation}
Thus, the displacement rank does not increase, and the residual can be
represented and updated from its generators in $\mathcal O((m+n)p)$
operations per iteration~\cite{gohberg1995fast,pan2000superfast,boros2002pivoting}.

Building on this property, \cite[Section~5.1]{GillesWilber2026} gives a fast
implementation of RPLU for Cauchy-like matrices. Exact RPLU pivots are
sampled using row-norm bounds from a Barnes--Hut-like algorithm,
followed by rejection sampling. The resulting method costs
$\mathcal O(kp^2(m+n)\log(m+n))$ operations for $k$ iterations, but requires
constructing and traversing spatial trees and interaction lists.

Approximate pivoting allows us to replace this procedure with a much cheaper
sampling distribution. At iteration $\ell$, define the generator row weights
\[
    w_i
    :=
    \norm{\mat{G}_{i,\col}^{(\ell)}\mat{B}^{(\ell)}}_2^2
    =
    \mat{G}_{i,\col}^{(\ell)}
    \bigl(\mat{B}^{(\ell)}(\mat{B}^{(\ell)})^{*}\bigr)
    (\mat{G}_{i,\col}^{(\ell)})^{*},
    \qquad i=1,\ldots,m.
\]
We define Random Generator pivoting (RGP) by sampling each pivot from the
following distribution:
\begin{equation}
    q_{ij}^{(\ell)}
    :=
    \Prob((i_{\ell+1},j_{\ell+1})=(i,j)\mid\mat{A}^{(\ell)})
    =
    \frac{w_i}{\norm{w}_1}
    \frac{\abs{\mat{A}_{ij}^{(\ell)}}^2}
         {\norm{\mat{A}_{i,\col}^{(\ell)}}_2^2}.
    \label{eq:generator-pivot-probability}
\end{equation}
Thus, RGP replaces the residual row norms with the generator row weights,
which require only a few vector operations and no spatial trees or geometric
precomputation.
After sampling a pivot according to~\eqref{eq:generator-pivot-probability},
the generators are updated using~\eqref{eq:cauchy-generator-update}. Each
iteration costs $\mathcal O((m+n)p^2)$ operations. The algorithm is
summarized in~\cref{alg:random-generator-pivoting}.

\begin{algorithm}[H]
\caption{LU with Random Generator Pivoting}
\label{alg:random-generator-pivoting}
\begin{algorithmic}[1]
\Require Generators $\mat{G}^{(0)}\in\mathbb C^{m\times p}$ and
$\mat{B}^{(0)}\in\mathbb C^{p\times n}$, nodes
$\mat{x}\in\mathbb C^m$ and $\mat{y}\in\mathbb C^n$ defining
$\mat{A}_{ij}=(\mat{G}^{(0)}\mat{B}^{(0)})_{ij}/(x_i-y_j)$, and target
rank $k$
\Statex \textbf{Output:} Index sets $I_k,J_k$ such that
$\mat{A}\approx
\mat{A}_{\col,J_k}\mat{A}_{I_k,J_k}^{-1}\mat{A}_{I_k,\col}$
\State $I_0\gets\varnothing$, $J_0\gets\varnothing$
\For{$\ell=0,\ldots,k-1$}
    \State $\mat{H}\gets
    \mat{B}^{(\ell)}(\mat{B}^{(\ell)})^{*}$
    \State $w\gets\diag\!\left(
    \mat{G}^{(\ell)}\mat{H}(\mat{G}^{(\ell)})^{*}\right)$
    \Comment{Compute row weights}
    \State Sample $i_{\ell+1}$ with probability
    $w_i/\norm{w}_1$
    \Comment{Sample row index}
    \State $\mat{r}_{\ell+1}\gets
    (\mat{G}_{i_{\ell+1},\col}^{(\ell)}\mat{B}^{(\ell)})
    \oslash(x_{i_{\ell+1}}-\mat{y})$
    \Comment{Form row $i_{\ell+1}$ of $\mat{A}^{(\ell)}$}
    \State Sample $j_{\ell+1}$ with probability
    $\abs{(\mat{r}_{\ell+1})_j}^2/\norm{\mat{r}_{\ell+1}}_2^2$
    \Comment{Sample column index}
    \State $\mat{c}_{\ell+1}\gets
    (\mat{G}^{(\ell)}\mat{B}_{\col,j_{\ell+1}}^{(\ell)})
    \oslash(\mat{x}-y_{j_{\ell+1}})$
    \Comment{Form column $j_{\ell+1}$ of $\mat{A}^{(\ell)}$}
    \State $p_{\ell+1}\gets(\mat{r}_{\ell+1})_{j_{\ell+1}}$
    \State $\mat{G}^{(\ell+1)}\gets
    \mat{G}^{(\ell)}-p_{\ell+1}^{-1}\mat{c}_{\ell+1}
    \mat{G}_{i_{\ell+1},\col}^{(\ell)}$
    \State $\mat{B}^{(\ell+1)}\gets
    \mat{B}^{(\ell)}-p_{\ell+1}^{-1}
    \mat{B}_{\col,j_{\ell+1}}^{(\ell)}\mat{r}_{\ell+1}$
    \Comment{Update generators}
    \State $I_{\ell+1}\gets I_\ell\cup\{i_{\ell+1}\}$,
    $J_{\ell+1}\gets J_\ell\cup\{j_{\ell+1}\}$
\EndFor
\State \Return $I_k,J_k$
\end{algorithmic}
\end{algorithm}

To compare RGP with exact RPLU, suppose that the Cauchy denominators satisfy
\[
  0<d_{\min}
  \leq\abs{x_i-y_j}
  \leq d_{\max}<\infty
\]
for every $i,j$. By~\eqref{eq:cauchy-like},
\[
  w_i\leq d_{\max}^2\norm{\mat{A}_{i,\col}^{(\ell)}}_2^2,
  \qquad
  \norm{w}_1\geq d_{\min}^2\norm{\mat{A}^{(\ell)}}_F^2.
\]
Substitution into~\eqref{eq:generator-pivot-probability} gives
\[
  q_{ij}^{(\ell)}
  \leq
  \left(\frac{d_{\max}}{d_{\min}}\right)^2
  \frac{\abs{\mat{A}_{ij}^{(\ell)}}^2}
       {\norm{\mat{A}^{(\ell)}}_F^2}.
\]
Consequently, RGP is $\gamma$-approximate RPLU, with
$\gamma=(d_{\min}/d_{\max})^2$, so RGP has the same convergence rates as RPLU, albeit with a potentially large prefactor.

In
\cref{fig:rgpruntime}, we compare RGP with the RPLU implementation
of~\cite{GillesWilber2026} and an RSVD accelerated by
FMM2D~\cite{AskhamEtAl2021}\footnote{The experiments were run on a 16-core
Apple M3 Max with 64~GB of memory using 12 threads.} on Loewner matrices, which are Cauchy-like with
displacement rank $p=2$. We observe that RGP and RPLU have nearly identical
approximation errors, both within a modest constant factor of optimal, while
RSVD is nearly optimal. RGP is nevertheless $16\times$ faster
than RPLU and $650\times$ faster than RSVD.

\begin{figure}
    \centering
    \includegraphics[
        width=\linewidth,
        trim=0 60 0 170,
        clip
    ]{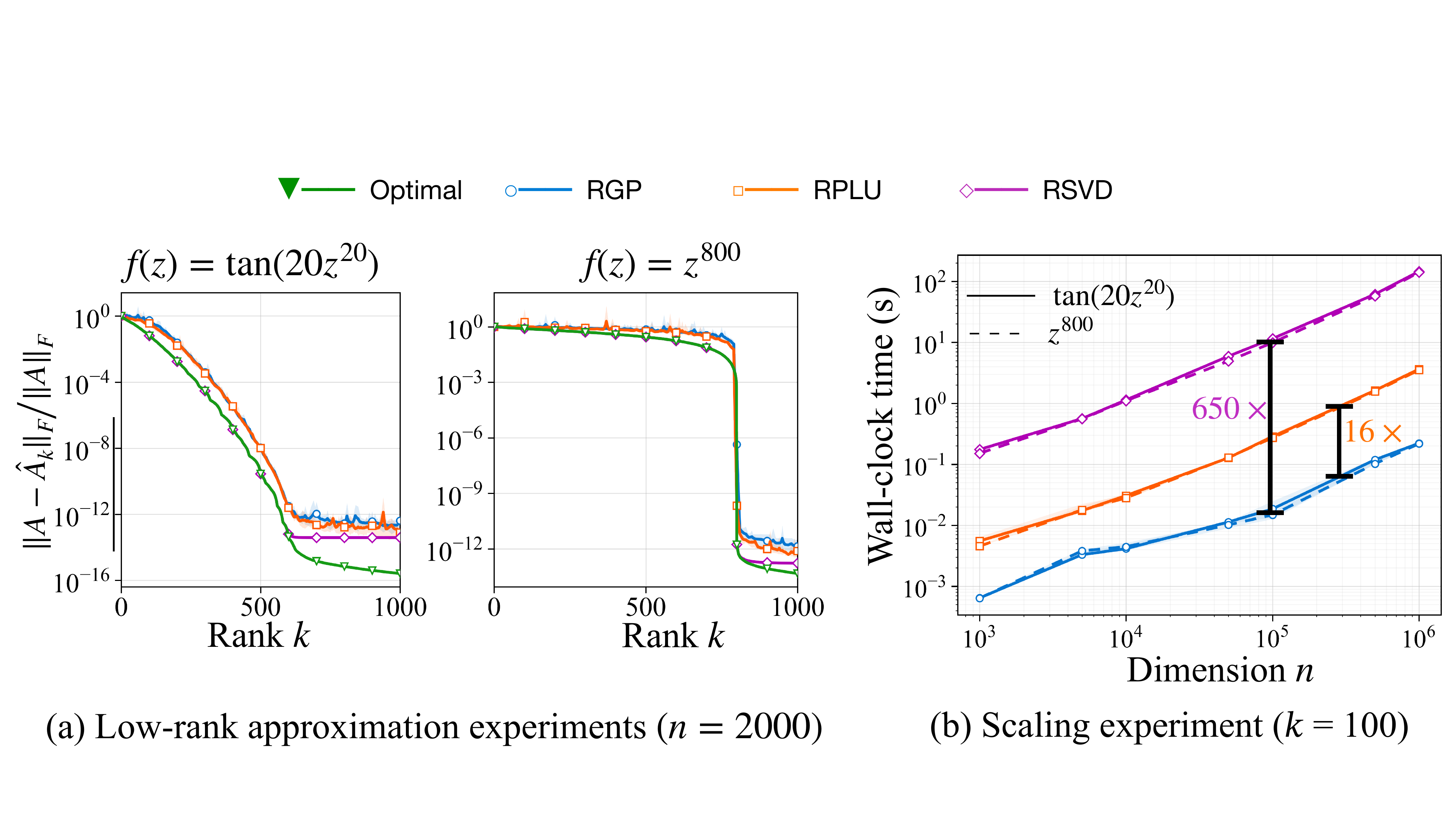}
    \caption{Approximation quality and timing of RGP, RPLU, and RSVD on two
    classes of Loewner matrices
    $\mat{A}_{ij}=(f(x_i)-f(y_j))/(x_i-y_j)$, where $x_i$ and $y_j$ are
    sampled independently and uniformly from the unit circle.
    Left: relative Frobenius error for $n=m=2000$, with
    $f(z)=\tan(20z^{20})$ and $f(z)=z^{800}$. RGP and RPLU are repeated ten
    times; we report the mean, with a shaded band showing the minimum and
    maximum over the ten runs. RSVD and the optimal approximation are
    computed once.
    Right: computation time as a function of $n=m$ for a rank-$100$
    approximation. }
    \label{fig:rgpruntime}
\end{figure}

\subsection{Sketchy pivoting}
\label{sec:sketchy-rplu}

Next, consider a sketch-based implementation of approximate RPLU,
motivated by recent work on randomized pivoting and CUR approximations
\cite{DongMartinsson2023,PearceEtAl2025,PritchardEtAl2025,ChenEtAl2020}.
The goal is to replace the expensive search over the full residual with a
search over a much smaller sketch.

To this end, fix a target rank $k$ and a block size $b\geq1$, and draw a random sketch
matrix partitioned into $k$ blocks,
\[
    \mat{\Omega}
    =
    [\mat{\Omega}_1\ \cdots\ \mat{\Omega}_k]
    \in\R^{n\times kb},
\]
where $\mat{\Omega}_1,\ldots,\mat{\Omega}_k\in\R^{n\times b}$. Then, compute
\[
    \mat{Y}^{(0)}
    =
    [\mat{Y}_1^{(0)}\ \cdots\ \mat{Y}_k^{(0)}]
    =
    \mat{A}\mat{\Omega},
\]
as a single matrix-matrix product.
The rest of the algorithm will rely on this single sketch to sample pivots.

At the first iteration, the first block of columns $\mat{Y}_1^{(0)}$ is used to estimate the row norms
and sample one pivot according to
\[
  \Prob((i_1,j_1)=(i,j)\mid\mat{\Omega}_1)
  =
  \frac{\norm{(\mat{Y}_1^{(0)})_{i,\col}}_2^2}
       {\norm{\mat{Y}_1^{(0)}}_F^2}
  \frac{\abs{\mat{A}_{ij}}^2}
       {\norm{\mat{A}_{i,\col}}_2^2},
\]
using the two stage sampling procedure described above. After the pivot is selected,
the sketch is downdated as
in~\cite{ChenEtAl2020,PritchardEtAl2025}
\[
\begin{aligned}
  \mat{Y}^{(1)}
  &=
  \mat{Y}^{(0)}
  -\frac{\mat{A}_{\col,j_1}}{\mat{A}_{i_1,j_1}}
   \mat{Y}^{(0)}_{i_1,\col}\\
  &=
  \left(
    \mat{A}
    -\frac{\mat{A}_{\col,j_1}\mat{A}_{i_1,\col}}
          {\mat{A}_{i_1,j_1}}
  \right)\mat{\Omega}
  =\mat{A}^{(1)}\mat{\Omega},
\end{aligned}
\]
so that it is a sketch of the
residual rather than the original matrix.

The remaining iterations proceed in the same way. At iteration $\ell$, the
active block satisfies
\[
    \mat{Y}_{\ell+1}^{(\ell)}
    =
    \mat{A}^{(\ell)}\mat{\Omega}_{\ell+1}.
\]
This block is used to estimate the row norms and sample a pivot, after which
the sketch is downdated for the next iteration.
The sequential use of the blocks is crucial for the analysis: since
$\mat{\Omega}_{\ell+1}$ is not used to sample pivots before iteration $\ell+1$, it remains
independent of $\mat{A}^{(\ell)}$ whenever the sketch blocks are independent.
For the computation, however, all sketching can be done up front in the
single block product $\mat{A}\mat{\Omega}$, and the rest of the algorithms accesses $\mat{A}$ only through the selected $k$ rows and $k$ columns of $\mat{A}$.

The algorithm is summarized in~\cref{alg:sketchy-rplu}, with the
approximant $\mat{\widehat A}^{(\ell)}$ from~\cref{eq:rplu-iteration}
stored in factored form as
$\mat{\widehat A}^{(\ell)}
=\mat{L}_{\col,1:\ell}\mat{U}_{1:\ell,\col}$.

\begin{algorithm}[H]
\caption{LU with sketchy pivoting}
\label{alg:sketchy-rplu}
\begin{algorithmic}[1]
\Require Matrix $\mat{A}\in\R^{m\times n}$, target rank $k$, block size
$b$, and sketch $\mat{\Omega}=[\mat{\Omega}_1\ \cdots\ \mat{\Omega}_k]
\in\R^{n\times kb}$
\Statex \textbf{Output:} Factors $\mat{L}\in\R^{m\times k}$ and
$\mat{U}\in\R^{k\times n}$ such that $\mat{A}\approx\mat{L}\mat{U}$
\State $\mat{Y}^{(0)}=[\mat{Y}_1^{(0)}\ \cdots\ \mat{Y}_k^{(0)}]
\gets\mat{A}\mat{\Omega}$
\State $\mat{L}\gets\mat{0}_{m\times k}$,
$\mat{U}\gets\mat{0}_{k\times n}$
\For{$\ell=0,\ldots,k-1$}
    \State Sample $i_{\ell+1}$ with probability
    $\norm{(\mat{Y}_{\ell+1}^{(\ell)})_{i,\col}}_2^2/
    \norm{\mat{Y}_{\ell+1}^{(\ell)}}_F^2$
    \Comment{Sample row index}
    \State $r_{\ell+1}\gets
    \mat{A}_{i_{\ell+1},\col}
    -\mat{L}_{i_{\ell+1},1:\ell}\mat{U}_{1:\ell,\col}$
    \Comment{Form row $i_{\ell+1}$ of $\mat{A}^{(\ell)}$}
    \State Sample $j_{\ell+1}$ with probability
    $\abs{(r_{\ell+1})_j}^2/\norm{r_{\ell+1}}_2^2$
    \Comment{Sample column index}
    \State $c_{\ell+1}\gets
    \mat{A}_{\col,j_{\ell+1}}
    -\mat{L}_{\col,1:\ell}\mat{U}_{1:\ell,j_{\ell+1}}$
    \Comment{Form column $j_{\ell+1}$ of $\mat{A}^{(\ell)}$}
    \State $p_{\ell+1}\gets(r_{\ell+1})_{j_{\ell+1}}$
    \State $\mat{L}_{\col,\ell+1}\gets c_{\ell+1}/p_{\ell+1}$,
    $\mat{U}_{\ell+1,\col}\gets r_{\ell+1}$
    \Comment{Update $\mat{L}$ and $\mat{U}$}
    \State $\mat{Y}^{(\ell+1)}\gets
    \mat{Y}^{(\ell)}
    -\mat{L}_{\col,\ell+1}\mat{Y}^{(\ell)}_{i_{\ell+1},\col}$
    \Comment{Downdate the sketch}\label{line:sketch-downdate}
\EndFor
\State \Return $\mat{L},\mat{U}$
\end{algorithmic}
\end{algorithm}

We now show that \cref{alg:sketchy-rplu} samples from an approximate RPLU
distribution. For the analysis, we take $\mat{\Omega}$ to be a standard
Gaussian matrix, that is, a matrix whose entries are independent standard
normal random variables.

Thus, it only remains to compare the row-sampling distribution induced by the
sketch with that defined by the exact row norms. High-probability estimates
of individual row norms under Gaussian projection is a standard consequence
of the Johnson--Lindenstrauss lemma; see, e.g., \cite{Woodruff2014}.
Here we require a normalized row norm estimate, and to obtain simple conditions, we give a dimension-independent bound on its expectation.

\begin{lemma}[Normalized Gaussian row norm estimates]
\label{lem:gaussian-row-energy-sampling}
Let $\mat{H}\in\R^{m\times n}$ be nonzero, let $b\geq3$, and let
$\mat{\Omega}\in\R^{n\times b}$ be a standard Gaussian matrix.
Then, for every $i$,
\begin{equation}
\label{eq:gaussian-row-energy-sampling}
    \E_{\mat{\Omega}}
    \left[
      \frac{\norm{\mat{H}_{i,\col}\mat{\Omega}}_2^2}
           {\norm{\mat{H}\mat{\Omega}}_F^2}
    \right]
    \leq
    \frac{b}{b-2}
    \frac{\norm{\mat{H}_{i,\col}}_2^2}
         {\norm{\mat{H}}_F^2}.
\end{equation}
\end{lemma}

The proof is deferred to~\cref{app:gaussian-row-norm-estimates}. We use it to establish the following proposition.

\begin{proposition}[Sketchy pivoting]
\label{prop:sketched-rplu}
Let \cref{alg:sketchy-rplu} use a block size $b\geq3$.
At iteration $\ell$, let
\[
    \mat{A}^{(\ell)}
    :=
    \mat{A}
    -\mat{L}_{\col,1:\ell}\mat{U}_{1:\ell,\col}
\]
be the current residual. The probability of sampling the pivot $(i,j)$ satisfies
\[
  q_{ij}^{(\ell)}
  :=
  \Prob\left(
    (i_{\ell+1},j_{\ell+1})=(i,j)
    \,\middle|\,
    \mat{A}^{(\ell)}
  \right)
  \leq
  \frac{b}{b-2}
  \frac{\abs{\mat{A}_{ij}^{(\ell)}}^2}
       {\norm{\mat{A}^{(\ell)}}_F^2}.
\]
\end{proposition}

\begin{proof}
Let $\mat{H}:=\mat{A}^{(\ell)}$ denote the current residual.
By the sketch downdate in line~\ref{line:sketch-downdate} of
\cref{alg:sketchy-rplu}, the active block satisfies
\[
    \mat{Y}_{\ell+1}^{(\ell)}
    =
    \mat{H}\mat{\Omega}_{\ell+1}.
\]
If $\norm{\mat{H}_{i,\col}}_2=0$, then $q_{ij}^{(\ell)}=0$ and the result is
immediate. Thus assume $\norm{\mat{H}_{i,\col}}_2>0$. For a fixed fresh
sketch block, the algorithm samples row $i$ with probability
\[
    \frac{\norm{\mat{H}_{i,\col}\mat{\Omega}_{\ell+1}}_2^2}
         {\norm{\mat{H}\mat{\Omega}_{\ell+1}}_F^2}
\]
and then column $j$, conditional on row $i$, with probability
$\abs{\mat{H}_{ij}}^2/\norm{\mat{H}_{i,\col}}_2^2$.
Thus,
\[
    \Prob\left(
      (i_{\ell+1},j_{\ell+1})=(i,j)
      \,\middle|\,
      \mat{H},
      \mat{\Omega}_{\ell+1}
    \right)
    =
    \frac{\norm{\mat{H}_{i,\col}\mat{\Omega}_{\ell+1}}_2^2}
         {\norm{\mat{H}\mat{\Omega}_{\ell+1}}_F^2}
    \frac{\abs{\mat{H}_{ij}}^2}
         {\norm{\mat{H}_{i,\col}}_2^2}.
\]
Averaging this probability over the fresh sketch block gives
\[
\begin{aligned}
    q_{ij}^{(\ell)}
    &=
    \Prob\left(
      (i_{\ell+1},j_{\ell+1})=(i,j)
      \,\middle|\,
      \mat{H}
    \right)\\
    &=
    \E_{\mat{\Omega}_{\ell+1}}
    \left[
      \Prob\left(
        (i_{\ell+1},j_{\ell+1})=(i,j)
        \,\middle|\,
        \mat{H},
        \mat{\Omega}_{\ell+1}
      \right)
    \right]\\
    &=
    \E_{\mat{\Omega}_{\ell+1}}
    \left[
        \frac{\norm{\mat{H}_{i,\col}\mat{\Omega}_{\ell+1}}_2^2}
             {\norm{\mat{H}\mat{\Omega}_{\ell+1}}_F^2}
    \right]
    \frac{\abs{\mat{H}_{ij}}^2}
         {\norm{\mat{H}_{i,\col}}_2^2}.
\end{aligned}
\]
The second equality uses the independence of $\mat{\Omega}_{\ell+1}$ and
$\mat{H}$.

By \cref{lem:gaussian-row-energy-sampling},
\[
\begin{aligned}
    q_{ij}^{(\ell)}
    &\leq
    \frac{b}{b-2}
    \frac{\norm{\mat{H}_{i,\col}}_2^2}{\norm{\mat{H}}_F^2}
    \frac{\abs{\mat{H}_{ij}}^2}
         {\norm{\mat{H}_{i,\col}}_2^2} \\
    &=
    \frac{b}{b-2}
    \frac{\abs{\mat{H}_{ij}}^2}
         {\norm{\mat{H}}_F^2}.
\end{aligned}
\]
\end{proof}

This proposition shows that Gaussian sketchy pivoting with $b=3$ is a $1/3$-approximate
RPLU scheme and therefore satisfies the same algebraic and geometric
convergence bounds as exact RPLU, up to a factor of $3$.

The cost of $k$ iterations of~\cref{alg:sketchy-rplu} is
\[
    T_{\mathrm{sketch}}(\mat{A},3k)
    +
    \mathcal O\bigl((m+n)k^2\bigr)
\]
operations, where $T_{\mathrm{sketch}}(\mat{A},3k)$ denotes the cost of
forming $\mat{A}\mat{\Omega}$ for a sketch matrix
$\mat{\Omega}\in\R^{n\times 3k}$, and thus has the same asymptotic complexity
as, for example, ~\cite[Algorithm~2]{DongMartinsson2023}.

However, \cref{alg:sketchy-rplu} differs from the sketch-based methods
in~\cite{DongMartinsson2023,PearceEtAl2025,PritchardEtAl2025} that motivated it in three important ways.
The first is the use of blocking (i.e. choose multiple pivots before accessing the original matrix),
and the second, in the case of~\cite{PritchardEtAl2025}, the use of small recycled sketches.
Neither can be covered by our current analysis, and likely requires substantially new ideas and assumptions.

The third is that~\cite{DongMartinsson2023,PearceEtAl2025,PritchardEtAl2025} choose pivots greedily from the
sketch, typically using partial pivoting, rather than randomly.
That is, they use a pivoting strategy more akin to C2PLU, defined as:
\[
    i_{\ell+1}
    \in
    \argmax_i
    \norm{\mat{A}^{(\ell)}_{i,\col}}_2,
    \qquad
    j_{\ell+1}
    \in
    \argmax_j
    \abs{\mat{A}^{(\ell)}_{i_{\ell+1},j}}.
\]
rather than approximate RPLU.\footnote{For example, with $b=1$, the partial pivoting option in
IterativeCUR~\cite[Algorithm~3.1]{PritchardEtAl2025} is equivalent to C2PLU with the row norms estimated from the
sketch.}

It is natural to ask whether the same analysis can cover this greedy pivoting
choice. Unfortunately, the direct bounds for C2PLU depend on the dimension.
Indeed, one can easily check that C2PLU is $n^{-1/2}$-approximate greedy
pivoting in the sense of~\cite{Gilles_convergence}, and
$(mn)^{-1}$-approximate RPLU in the sense
of~\cref{eq:approximate-sampling}. Thus, it inherits both convergence rates,
but with dimension-dependent prefactors.

\section{Conclusion}

We established algebraic and geometric convergence rates for approximate
RPCholesky, RPQR, and RPLU. The main tools are the pivot-product identity and
estimates of elementary symmetric polynomials. Approximate sampling changes
the bounds only by the prefactor $\gamma^{-1}$, allowing the same analysis to
cover the two cheaper variants of RPLU introduced here.

These results naturally raise the question of whether greedy or randomized
pivoting is preferable.
Under algebraic spectral decay, RPCholesky and RPQR achieve the same
asymptotic rates as their greedy counterparts. These rates also match the
optimal trace and squared Frobenius errors, respectively.
The comparison between RPLU and its greedy counterpart is less direct. Let
$\llbracket\mat{A}\rrbracket_k$ denote the rank-$k$ truncated SVD of
$\mat{A}$. In the model case
$\sigma_j(\mat{A})=j^{-\alpha}$, with $\alpha>1$, our RPLU result gives
\[
    \E\min_{0\leq\ell\leq k}
    \norm{\mat{A}-\mat{\widehat A}^{(\ell)}}_F^2
    =
    \mathcal{O}(k^{2-2\alpha}),
    \qquad
    \norm{\mat{A}-\llbracket\mat{A}\rrbracket_k}_F^2
    =
    \Theta(k^{1-2\alpha}).
\]
Thus, RPLU loses one factor of $k$ relative to the optimal squared Frobenius
error. By comparison, the greedy CPLU result from~\cite{Gilles_convergence}
is
\[
\begin{aligned}
    \min_{0\leq\ell\leq k}
    \norm{\mat{A}-\mat{\widehat A}^{(\ell)}}_{\max}
    &=
    \mathcal{O}(k^{-\alpha}),\\
    \norm{\mat{A}-\llbracket\mat{A}\rrbracket_k}_{\max}
    &\leq
    \sigma_{k+1}(\mat{A})
    =
    (k+1)^{-\alpha}
    \leq
    \sqrt{mn}\,
    \norm{\mat{A}-\llbracket\mat{A}\rrbracket_k}_{\max}.
\end{aligned}
\]
Thus, CPLU follows the singular-value decay without a loss factor, but in max norm; and
Frobenius-norm error can be as much as a factor $\sqrt{mn}$ larger than the
max-norm error.

The bounds under geometric decay are less satisfactory for both randomized
and greedy methods. If the singular values decay as
$\mathcal{O}(\rho^j)$, with $0<\rho<1$, then the optimal rank-$k$ error also decays as
$\mathcal{O}(\rho^k)$, whereas our bounds give
$\mathcal{O}(k\rho^{k/2})$ RPCholesky and
RPLU, and
$\mathcal{O}(\sqrt{k}\rho^{k/2})$ for RPQR.

The source of this gap is clear in the proof. Our analysis, as well as the one
in~\cite{Gilles_convergence}, relies on the pivot-product
identity, which controls a product of the first $k$ residual errors. To obtain
a bound for a single iterate, we compare its error with the geometric mean of
the preceding errors. Under geometric spectral decay, this step gives a rate
of $\rho^{k/2}$ rather than $\rho^k$. In practice, however, the observed
errors appear to follow the same $\mathcal{O}(\rho^k)$ rate as the optimal
approximation.

The practical evidence on greedy versus randomized pivoting is similarly
mixed. Randomized pivoting can be more robust to outliers
\cite{ChenEtAl2025,GillesWilber2026}, whereas experiments on sparse matrices sometimes find greedy pivoting more
accurate \cite{GillesWilber2026,fornace2024column}.

Randomization nevertheless has a distinct advantage in the present
analysis: estimates of the residual row norms can be converted directly into
convergence guarantees for approximate randomized pivoting, as done in~\cref{sec:rgp,sec:sketchy-rplu}.
No such dimension-independent guarantees currently exist for greedy pivoting.

Several open questions remain. Can the geometric-mean loss in the
pivot-product analysis be removed, recovering the
$\mathcal{O}(\rho^k)$ rate observed in practice? Under what assumptions can
the analysis accommodate blocked pivot selection or a recycled
sketch? Can greedy sketch-based pivoting satisfy dimension-independent
guarantees?

\section*{AI Acknowledgement}

The extension to sketchy pivoting was simplified through an
interaction with OpenAI's GPT-5.6 Sol Pro. The model was also used to
assist with writing and with producing figures and code. All other
results were proved by the authors, based on the results for greedy
pivoting in~\cite{Gilles_convergence}.

\appendix

\section{Proof of the elementary symmetric polynomial estimates}
\label[appendix]{app:esp-decay}

\begin{proof}[Proof of \cref{lem:esp-decay}]
For nonnegative arguments, $e_k$ is componentwise nondecreasing. Thus,
\[
  e_k(\lambda_1,\ldots,\lambda_n)
  \leq
  C^k e_k(1,2^{-\alpha},\ldots,n^{-\alpha}).
\]
For $t>0$, consider the generating function for $e_k$
\cite[Section~1.2, Equation~2.2]{Macdonald1998}:
\[
  \prod_{j=1}^n(1+t j^{-\alpha})
  =\sum_{r=0}^n t^r e_r(1,2^{-\alpha},\ldots,n^{-\alpha}).
\]

Since every term in the sum is nonnegative and $e_k$ appears in the $k$th term,
\[
  e_k(1,2^{-\alpha},\ldots,n^{-\alpha})
  \leq
  t^{-k}\prod_{j=1}^n(1+t j^{-\alpha}).
\]
We bound the log of the product on the right-hand side. Since the integrand is decreasing:
\begin{align*}
  \sum_{j=1}^n\log(1+t j^{-\alpha})
  &\leq\int_0^\infty\log(1+t x^{-\alpha})\,\mathrm{d}x\\
  &=t^{1/\alpha}\int_0^\infty\log(1+u^{-\alpha})\,\mathrm{d}u\\
  &=t^{1/\alpha}\frac{\pi}{\sin(\pi/\alpha)},
\end{align*}
where the second line follows from the change of variables
$u=t^{-1/\alpha}x$, and the third from the identity
\[
  \int_0^\infty\log(1+u^{-\alpha})\,\mathrm{d}u
  =\frac{\pi}{\sin(\pi/\alpha)},
\]
see~\cite[Formula~4.293.3]{GradshteynRyzhik2007}.

Thus, for all $t>0$,
\[
  e_k(\lambda_1,\ldots,\lambda_n)
  \leq
  C^k t^{-k}\exp\left(\frac{\pi t^{1/\alpha}}{\sin(\pi/\alpha)}\right).
\]
Setting the derivative of the right-hand side to zero shows that it is minimized at
\[
  t=\left(\frac{\alpha k\sin(\pi/\alpha)}{\pi}\right)^\alpha.
\]
Substituting this value of $t$ and taking the $k$th roots gives
\[
  e_k(\lambda_1,\ldots,\lambda_n)^{1/k}
  \leq
  C\left(\frac{\pi\e}{\alpha k\sin(\pi/\alpha)}\right)^\alpha.
\]

For the geometric estimate, we use the fact that $e_k$ is nondecreasing under additional arguments when all arguments are nonnegative, so
\[
  e_k(\lambda_1,\ldots,\lambda_n)
  \le e_k(C,C\rho,\ldots,C\rho^{n-1})
  \leq C^ke_k(\rho^0,\rho^1, \rho^2, \ldots).
\]
Then, the $q$-binomial identity \cite[Section~1.2, Example~4]{Macdonald1998} gives

\[
  e_k(\rho^0,\rho^1, \rho^2, \ldots)
  = \frac{\rho^{k(k-1)/2}}{\prod_{j=1}^k(1-\rho^j)}
  \leq \left(\frac{\rho^{(k-1)/2}}{1-\rho}\right)^k.
\]

Using these two upper bounds and taking the $k$th root gives
\[
e_k(\lambda_1,\ldots,\lambda_n)^{1/k} \le \frac{C}{1-\rho} \rho^{(k-1)/2}.
\]
\end{proof}

\section{Proof of the Gaussian row norm estimates}
\label[appendix]{app:gaussian-row-norm-estimates}

\begin{proof}[Proof of \cref{lem:gaussian-row-energy-sampling}]
Set
\[
    w_j:=\norm{\mat{H}_{j,\col}}_2^2,
    \qquad
    \widehat w_j:=\norm{\mat{H}_{j,\col}\mat{\Omega}}_2^2,
\]
and
\[
    W:=\sum_jw_j=\norm{\mat{H}}_F^2,
    \qquad
    \widehat W:=\sum_j\widehat w_j
    =\norm{\mat{H}\mat{\Omega}}_F^2.
\]
If $w_i=0$, then $\mat{H}_{i,\col}=0$ and the result is immediate for that
index. Removing all such zero rows leaves both denominators unchanged, so
assume $w_i>0$ for every $i$.

Define
\[
    u_j:=\frac{\mat H_{j,\col}^{\top}}{\sqrt{w_j}},
    \qquad
    X_j:=\frac{\widehat w_j}{w_j}
        =\norm{\mat\Omega^{\top}u_j}_2^2.
\]
Thus, \(\norm{u_j}_2=1\), \(\widehat w_j=w_jX_j\), and
\(X_j\sim\chi_b^2\).

By Cauchy--Schwarz,
\[
\begin{aligned}
    W^2
    &=
    \left(
        \sum_{j}
        \sqrt{w_jX_j}\sqrt{\frac{w_j}{X_j}}
    \right)^2 \\
    &\leq
    \left(\sum_{j} w_jX_j\right)
    \left(\sum_{j} \frac{w_j}{X_j}\right)
    =
    \widehat W
    \sum_{j} \frac{w_j}{X_j}.
\end{aligned}
\]
Using this upper bound and $\widehat w_i=w_iX_i$ gives
\[
    \E_{\mat{\Omega}}\frac{\widehat w_i}{\widehat W}
    \leq
    \frac{w_i}{W^2}
    \sum_{j}
    w_j\,\E_{\mat{\Omega}}\frac{X_i}{X_j}.
\]

To bound these ratios, fix $j$ and set
$\rho:=u_i^\top u_j$. The variables $X_i$ and $X_j$ are obtained by summing
the squares of $b$ pairs of standard Gaussian variables with correlation
$\rho$. Applying
\cite[Corollary~3.8(i)]{Joarder2009} with $m=b$ to
$X_i/X_j$ gives
\[
    \E_{\mat{\Omega}}\frac{X_i}{X_j}
    =
    \frac{b-2\rho^2}{b-2}
    \leq
    \frac{b}{b-2}.
\]

Substituting this estimate above gives
\[
\begin{aligned}
    \E_{\mat{\Omega}}\frac{\widehat w_i}{\widehat W}
    &\leq
    \frac{b}{b-2}
    \frac{w_i}{W^2}
    \sum_{j} w_j \\
    &=
    \frac{b}{b-2}\frac{w_i}{W},
\end{aligned}
\]
as claimed.
\end{proof}

\bibliographystyle{plain}
\bibliography{references}

\end{document}